\documentclass[twoside]{article}
\usepackage{amsmath,amsthm,amssymb}
\usepackage{newtxtext,newtxmath}
\usepackage[text={12.5cm,17.2cm},centering,paperwidth=16.96cm,paperheight=23.96cm]{geometry}
\usepackage[fontsize=10.4pt]{scrextend}
\usepackage[hyperfootnotes=false,colorlinks=true,allcolors=blue]{hyperref}
\usepackage{array}
\usepackage{multirow}
\usepackage{tabularx}

\def\titlerunning#1{\gdef\titrun{#1}}

\makeatletter
\def\author#1{\gdef\autrun{\def\and{\unskip, }#1}\gdef\@author{#1}}

\makeatother

\def\email#1{E-mail:\href{mailto:#1}{#1}}
\def\subjclass#1{\par\bigskip\noindent\textsl{Mathematics Subject Classification 2020.} #1}
\def\keywords#1{\par\smallskip\noindent\textsl{Keywords and Phrases.} #1}

\newenvironment{dedication}{\itshape\center}{\par\medskip}

\newtheorem{thm}{Theorem}[section]

\newtheorem{exm}[thm]{Example}
\newtheorem{lem}[thm]{Lemma}

\theoremstyle{definition}

\numberwithin{equation}{section}

\newcommand{\Q}{\mathbb{Q}}
\newcommand{\R}{\mathbb{R}}

\begin{document}

% Give an abbreviation of the title for the running page headers.
\titlerunning{A new proof of Legendre's theorem on the Diophantine equation $ax^2+by^2+cz^2=0$}

% Here you can enter the full article title.
\title{\textbf{A new proof of Legendre's theorem on the Diophantine equation $ax^2+by^2+cz^2=0$}}

% Here you can enter the full names of authors separated by \and.
\author{Jingbo Liu and Bruce McOsker}

% Please do not enter a date.
\date{}

\maketitle

% Here you can enter an optional dedication.
\begin{dedication}
Department of Mathematical, Physical, and Engineering Sciences,\\ Texas A\&M University-San Antonio, San Antonio, Texas, USA\\
\email{\,jliu@tamusa.edu} (Jingbo Liu; faculty advisor)\\ \email{\,mcoskerbruce@gmail.com} (Bruce McOsker; undergraduate student)
\end{dedication}

% Here you can enter the address and email of each author separated by \and by following the example.
%\address{Department of Mathematical, Physical, and Engineering Sciences, Texas A\&M University-San Antonio, San Antonio, TX, USA; \email{jliu@tamusa.edu}}



\begin{abstract}
One of Legendre's theorems on the Diophantine equation $ax^2+by^2+cz^2=0$ provides necessary and sufficient conditions on the existence of nonzero rational solutions of this equation, which helps determine the existence of rational points on a conic.
In this paper, we give a new proof of this famous theorem using Hasse invariant and Jacobi symbol from the theory of quadratic forms.
\end{abstract}

% Here you can enter the abstract, MSC classes, and keywords.
%\begin{abstract}
%An abstract.
%\subjclass{Primary 12X34; Secondary 56Y78}
%\keywords{Aaaa, bbbb, cccc}
%\end{abstract}

%%%%%%%%%%%%%%%%%%%%%%%%%%%%%%%%%%%%%%%%%%%%%%%%%%%%%%%%%%%%%%%%%%%%%%%%%%%%%%%%%%%%%%%%%%%%%%%%%%%%%%%%%%%%%%%%%%%%%%%%%%%%%%%%%%%%%%%%%%%%%%%%%%%%%%%%%%%%%%%%%%%%%%%%%%%%%%%%%
%%%%%%%%%%%%%%%%%%%%%%%%%%%%%%%%%%%%%%%%%%%%%%%%%%%%%%%%%%%%%%%%%%%%%%%%%%%%%%%%%%%%%%%%%%%%%%%%%%%%%%%%%%%%%%%%%%%%%%%%%%%%%%%%%%%%%%%%%%%%%%%%%%%%%%%%%%%%%%%%%%%%%%%%%%%%%%%%%
\section{Introduction}\label{Sec:Int}
In algebraic geometry, a conic defined by a quadratic polynomial equation
\begin{equation}\label{E1}
AX^2+BXY+CY^2+DX+EY+F=0,
\end{equation}
with $A,B,C,D,E,F$ rational numbers, has either no or infinitely many rational points on it.
Once a rational point on this conic is found, all the other rational points on it can be described using slope parametrization.
So, the key issue is how to find one rational point, or at least how to determine if there is any rational point, on the conic.

Replacing $X$ by $\frac{X}{Z}$ and $Y$ by $\frac{Y}{Z}$ in equation \eqref{E1}, one has a homogenous quadratic equation
\begin{equation}\label{E2}
AX^2+BXY+CY^2+DXZ+EYZ+FZ^2=0;
\end{equation}
moreover, one knows from linear algebra that the quadratic form $AX^2+BXY+CY^2+DXZ+EYZ+FZ^2$ in equation \eqref{E2} can be transformed to a diagonal quadratic form $ax^2+by^2+cz^2$ by replacing $X,Y,Z$ by linear forms of $x,y,z$ with rational coefficients.
Hence, the conic \eqref{E1} has a rational point if and only if the quadratic equation
\begin{equation}\label{E3}
ax^2+by^2+cz^2=0
\end{equation}
has a nonzero solution in $\Q^3$.

It is natural to assume $a,b,c$ in equation \eqref{E3} are nonzero square-free integers with $(a,b,c)=1$.
Assume further $a,b,c$ are pairwise relatively prime, since if, for instance, $a,b$ have a common divisor $d$, then $ax^2+by^2+cz^2=0$ has a nonzero solution if and only if $(a/d)x^2+(b/d)y^2+cdz^2=0$ has one.

Now, for the existence of a nonzero solution of equation \eqref{E3}, Legendre proved a renowned result in 1785 as follows.

\begin{thm}\label{LegendreTheorem}
Let $a,b,c$ be nonzero square-free integers such that $(a,b)=(b,c)=(c,a)=1$.
Then, the Diophantine equation \eqref{E3} has a nonzero solution in $\Q^3$ if and only if one has
\vskip 2pt
\noindent{\bf(A)} $a,b,c$ do not have the same sign;
\vskip 2pt
\noindent{\bf(B)} $-bc,-ca,-ab$ are quadratic residues modulo $a,b,c$, respectively.
\end{thm}

Due to its close relation with rational points on conics, this theorem has piqued the interest of many distinguished mathematicians for centuries, including C.F. Gauss and P.G.L. Dirichlet, who either reproved it using new techniques, or found an upper bound for its small nonzero solutions, or designed algorithms for calculating all its solutions in certain cases, or studied analogous problems on other number fields.
There is a large literature on this topic, and some details can be found, for example, in the monographs \cite[Pages 89-92]{Cassels}, \cite{Dirichlet}, \cite[Pages 294-298]{Gauss}, \cite[Chapter 5]{Grosswald} as well as \cite[Chapter IV and Appendix I]{Weil}, in the work \cite{CochraneMitchell}, \cite{CremonaRusin}, \cite{DavenportHall}, \cite{Holzer}, \cite{HudsonWilliams}, \cite{Kneser}, \cite{Leal-Ruperto},
%\cite{Leal-RupertoLeep},
\cite{Mordell1,Mordell2}, \cite{Samet}, \cite{Simon}, \cite{Skolem},
%\cite{HoeijCremona},
\cite{Watkins} and \cite{Williams}, and in their references (as far as the first author could verify).

In this paper, we provide another proof of this classical result using methods from the modern theory of quadratic forms.
It is easy to see the necessity of conditions {\bf(A)}-{\bf(B)}, but it is nontrivial to get the sufficiency of these conditions.
For our approach, we shall prove this theorem through the Hasse-Minkowski principle.
It is noteworthy that although Hasse invariant is often used to determine the existence of nonzero solutions of some given quadratic form over $\Q$, we nevertheless apply a relation between Hilbert symbol and Jacobi symbol to work out the case $p=2$ independently without resorting to Hilbert’s Reciprocity Law.
This work is a directed undergraduate research outcome and accessible to students of some basic knowledge in number theory.

%%%%%%%%%%%%%%%%%%%%%%%%%%%%%%%%%%%%%%%%%%%%%%%%%%%%%%%%%%%%%%%%%%%%%%%%%%%%%%%%%%%%%%%%%%%%%%%%%%%%%%%%%%%%%%%%%%%%%%%%%%%%%%%%%%%%%%%%%%%%%%%%%%%%%%%%%%%%%%%%%%%%%%%%%%%%%%%%%
%%%%%%%%%%%%%%%%%%%%%%%%%%%%%%%%%%%%%%%%%%%%%%%%%%%%%%%%%%%%%%%%%%%%%%%%%%%%%%%%%%%%%%%%%%%%%%%%%%%%%%%%%%%%%%%%%%%%%%%%%%%%%%%%%%%%%%%%%%%%%%%%%%%%%%%%%%%%%%%%%%%%%%%%%%%%%%%%%
\section{Hasse invariant and Jacobi symbol}\label{Sec:Hasse&Jacobi}
For a given ternary quadratic form $f(x,y,z)=ax^2+by^2+cz^2$, the renowned Hasse-Minkowski principle says that the Diophantine equation $f(x,y,z)=0$ has a nonzero solution in $\Q^3$ if and only if it has a nonzero solution in $\R^3$ and $\Q_p^3$ for every prime $p$, where $\Q_p$ denotes the completion of $\Q$ with respect to the $p$-adic valuation.
Hence, it is sufficient to prove that, under the conditions {\bf(A)}-{\bf(B)} of Theorem \ref{LegendreTheorem}, $f(x,y,z)=0$ has a nonzero solution locally everywhere.

In this section, we use Hasse invariant to determine if $f(x,y,z)=0$ has a nonzero solution in $\Q_p^3$.
Let $m,n$ be two nonzero integers; for each prime $p$, the Hilbert symbol $(m,n)_p$ is defined as $1$ if $mx^2+ny^2=1$ has a solution $(x,y)\in\Q_p^2$ and $-1$ otherwise.
Accordingly, the Hasse invariant for $f$ at $p$ is defined to be
\begin{equation}\label{Hasse}
S_p(f):=(a,-1)_p(b,-1)_p(c,-1)_p(a,b)_p(b,c)_p(c,a)_p.
\end{equation}
Similarly, one can define the Hilbert symbol $(m,n)_\infty$ and the Hasse invariant $s_\infty(f)$ in $\mathbb{R}$; in particular, $(m,n)_\infty=1$ if either $m$ or $n$ is positive and $(m,n)_\infty=-1$ otherwise.
Recall that Hilbert's reciprocity law says $(m,n)_\infty\cdot\prod_p(m,n)_p=1$, which further leads to $s_\infty(f)\cdot\prod_ps_p(f)=1$.

It is well-known $f(x,y,z)=0$ has a nonzero solution in $\Q_p^3$ if and only if $S_p(f)=(-1,-1)_p$; that is, $s_p(f)=-1$ if $p=2$ and $s_p(f)=1$ otherwise.
Some properties of the Hilbert symbols (listed as the first three lemmata without proofs below) enable us to calculate the values of them.
Interested reader can consult the classical monograph \cite[\S63B and \S73]{OMeara} for details.

\begin{lem}\label{L1}
Let $k,m,n$ be nonzero integers.
Then, one has
\vskip 2pt
\noindent{\sc(1-a)} $(m,n)_p=(n,m)_p$ and $\big(m,n^2\big)_p=1$,
\vskip 2pt
\noindent{\sc(1-b)} $(m,m)_p=(m,-1)_p$,
\vskip 2pt
\noindent{\sc(1-c)} $(k,mn)_p=(k,m)_p(k,n)_p$.
\end{lem}

\begin{lem}\label{L2}
Let $p\geq3$ be a prime, and let $m,n$ be integers with $(m,p)=(n,p)=1$.
Then, one has
\vskip 2pt
\noindent{\sc(2-a)} $(m,n)_p=1$,
\vskip 2pt
\noindent{\sc(2-b)} $(m,p)_p=\Big(\displaystyle\frac{m}{p}\Big)$.
\vskip 2pt
\noindent Here, $\Big(\displaystyle\frac{m}{p}\Big)$ denotes the Legendre symbol defined as $1$ if $m$ is a quadratic residue modulo $p$ and $-1$ otherwise; note $\Big(\displaystyle\frac{m}{p}\Big)=0$ if $p|m$.
\end{lem}

\begin{lem}\label{L3}
Let $m,n$ be odd integers.
Then, for $p=2$, one has
\vskip 2pt
\noindent{\sc(3-a)} $(m,n)_2=\displaystyle(-1)^{\frac{(m-1)(n-1)}{4}}=\biggl\{\begin{array}{ll}
-1&\text{if}\,\,m\equiv n\equiv3~(\bmod~4)\medskip\\
1&\text{otherwise}\end{array},$
\vskip 2pt
\noindent{\sc(3-b)} $(m,2)_2=\displaystyle(-1)^{\frac{(m-1)(m+1)}{8}}
=\biggl\{\begin{array}{ll}
-1&\text{if}\,\,m\equiv3,5~(\bmod~8)\medskip\\
1&\text{otherwise}\end{array}.$
\end{lem}

Let $n$ be a positive odd integer with prime factorization $n=p_1^{\alpha_1}p_2^{\alpha_2}\cdots p_k^{\alpha_k}$.
For every integer $a$, the Jacobi symbol $\Big(\displaystyle\frac{a}{n}\Big)$ is defined to be the product of the associated Legendre symbols
\begin{equation}\label{Jacobi}
\Big(\frac{a}{n}\Big):=\Big(\frac{a}{p_1}\Big)^{\alpha_1}\Big(\frac{a}{p_2}\Big)^{\alpha_2}\cdots\Big(\frac{a}{p_k}\Big)^{\alpha_k}.
\end{equation}

We list some well-known properties of the Jacobi symbols below.
Interested reader can consult the classical monograph \cite[Pages 56-57]{IrelandRosen} for details.
\begin{lem}\label{L4}
Let $a,b$ be integers, and let $m,n$ be positive odd integers with $(m,n)=1$.
Then, one has
\vskip 2pt
\noindent{\sc(4-a)} $\displaystyle\Big(\frac{a}{m}\Big)=\pm1$ if $(a,m)=1$ and $\displaystyle\Big(\frac{a}{m}\Big)=0$ otherwise,
\vskip 2pt
\noindent{\sc(4-b)} $a\equiv b~(\bmod~m)$ leads to $\displaystyle\Big(\frac{a}{m}\Big)=\Big(\frac{b}{m}\Big)$,
\vskip 2pt
\noindent{\sc(4-c)} $\displaystyle\Big(\frac{ab}{m}\Big)=\Big(\frac{a}{m}\Big)\Big(\frac{b}{m}\Big)$ and $\displaystyle\Big(\frac{a}{mn}\Big)=\Big(\frac{a}{m}\Big)\Big(\frac{a}{n}\Big)$,
\vskip 2pt
\noindent{\sc(4-d)} $\displaystyle\Big(\frac{-1}{m}\Big)=(-1)^{\frac{(m-1)}{2}}$ and $\displaystyle\Big(\frac{2}{m}\Big)=(-1)^{\frac{(m-1)(m+1)}{8}}$,
\vskip 2pt
\noindent{\sc(4-e)} $\displaystyle\Big(\frac{m}{n}\Big)\displaystyle\Big(\frac{n}{m}\Big)=(-1)^{\frac{(m-1)(n-1)}{4}}$.
\end{lem}

%%%%%%%%%%%%%%%%%%%%%%%%%%%%%%%%%%%%%%%%%%%%%%%%%%%%%%%%%%%%%%%%%%%%%%%%%%%%%%%%%%%%%%%%%%%%%%%%%%%%%%%%%%%%%%%%%%%%%%%%%%%%%%%%%%%%%%%%%%%%%%%%%%%%%%%%%%%%%%%%%%%%%%%%%%%%%%%%%
%%%%%%%%%%%%%%%%%%%%%%%%%%%%%%%%%%%%%%%%%%%%%%%%%%%%%%%%%%%%%%%%%%%%%%%%%%%%%%%%%%%%%%%%%%%%%%%%%%%%%%%%%%%%%%%%%%%%%%%%%%%%%%%%%%%%%%%%%%%%%%%%%%%%%%%%%%%%%%%%%%%%%%%%%%%%%%%%%
\section{Proof of Theorem \ref{LegendreTheorem}}\label{Sec:Proof}
In this section, we reprove Theorem \ref{LegendreTheorem}.
Let $f(x,y,z)=ax^2+by^2+cz^2$ be the ternary quadratic form in equation \eqref{E3}.
By the Hasse-Minkowski principle, it suffices to check that the Diophantine equation $f(x,y,z)=0$ has a nonzero solution in $\Q_p^3$ and $\R^3$ when the coefficients $a,b,c$ satisfy the conditions {\bf(A)}-{\bf(B)}.

\begin{proof}
The proof of necessity is trivial.
Actually, recall $a,b,c$ are nonzero, square-free, and pairwise relatively prime.
First, $a,b,c$ must be of different signs, as otherwise one sees either $f>0$ or $f<0$ for each nonzero triple $(x,y,z)$ in $\Q^3$.
Next, assume that the equation $f(x,y,z)=0$ has a nonzero solution $(x_0,y_0,z_0)\in\Q^3$, and choose $x_0,y_0,z_0$ suitably so that they are pairwise relatively prime.
Note that $c,x_0$ are relatively prime.
Suppose on the contrary $c,x_0$ have a common prime divisor $d$; then, $d$ divides $by_0^2$, a contradiction in view of $(b,c)=1$ or $(x_0, y_0)=1$.
Thus, $ax_0^2+by_0^2+cz_0^2=0$ leads to $-ab\equiv\big(by_0x_0^{-1}\big)^2~(\bmod~c)$, with $x_0^{-1}$ being the inverse of $x_0$ modulo $c$.
Similarly, one sees $-bc,-ca$ are quadratic residues modulo $a,b$, respectively.

Now, we prove the sufficiency.
Condition {\bf(A)} certainly guarantees the existence of a nonzero solution in $\R^3$.
Besides, if $p$ is odd and does not divide $abc$, then it is easy to see $S_p(f)=1=(-1,-1)_p$ which further shows equation \eqref{E3} has a nonzero solution in $\Q_p$.
Therefore, we only need to verify the case where $p=2$ and the cases where $p$ is odd but $p$ divides exactly one of $a,b,c$ (as they are pairwise relatively prime).

First, consider the cases where $p$ is odd.
Without loss of generality, assume that $p$ divides $a$ but not $b,c$.
Write $a=\tilde{a}p$ with $(\tilde{a},p)=1$ seeing $a$ is square-free.
Then, by Lemmata \ref{L1}-\ref{L2}, one has
\begin{equation}\label{Hasse*}
\begin{split}
s_p(f)&=(a,-1)_p(b,-1)_p(c,-1)_p(a,b)_p(b,c)_p(c,a)_p\\
      &=(a,-bc)_p=(\tilde{a},-bc)_p(p,-bc)_p\\
      &=(p,-bc)_p=1=(-1,-1)_p
\end{split}
\end{equation}
as condition {\bf(B)} implies that $-bc$ is a quadratic residue of $p$.
Therefore, $s_p(f)=1=(-1,-1)_p$ when $p$ is odd.

When $p=2$, we observe $s_2(f)=(-1,-1)_2$ immediately via Hilbert's reciprocity law and the fact $s_\infty(f)=-1=(-1,-1)_\infty$.
Here, however, we would like to introduce another presumably new approach based on the relation between Hilbert symbol and Jacobi symbol.
We have the identities
\begin{equation*}
(m,-1)_2=\Big(\frac{-1}{m}\Big),\,\,(m,2)_2=\Big(\frac{2}{m}\Big)\,\,\text{and}\,\,(m,n)_2=\Big(\frac{m}{n}\Big)\Big(\frac{n}{m}\Big)
\end{equation*}
for coprime positive odd integers $m,n$, which follow from Lemmata \ref{L3}-\ref{L4}.

We will only consider the case where $a<0,b<0$ yet $c>0$, because the proofs of the other cases are similar.
Note, however, slightly different congruence relations will be derived when $a,b,c$ have different signs.

When $a,b,c$ are all odd, by condition {\bf(B)}, there are integers $u,v,w$ such that
\begin{equation*}
-ab\equiv u^2~(\bmod~c),\,-bc\equiv v^2~(\bmod-a)\,\,\text{and}\,-ca\equiv w^2~(\bmod-b).
\end{equation*}
These congruence relations imply that
\begin{equation*}
\begin{split}
&\Big(\frac{-1}{c}\Big)\Big(\frac{-a}{c}\Big)\Big(\frac{-b}{c}\Big)=\Big(\frac{u^2}{c}\Big)=1,\\
&\Big(\frac{-b}{-a}\Big)\Big(\frac{c}{-a}\Big)=\Big(\frac{v^2}{-a}\Big)=1\\
&\,\text{and}\,\,\Big(\frac{-a}{-b}\Big)\Big(\frac{c}{-b}\Big)=\Big(\frac{w^2}{-b}\Big)=1.
\end{split}
\end{equation*}
Multiplying both sides of these three equations, it follows that
\begin{equation*}
\Big(\frac{-1}{c}\Big)\Big(\frac{-a}{-b}\Big)\Big(\frac{-b}{-a}\Big)\Big(\frac{-a}{c}\Big)\Big(\frac{c}{-a}\Big)\Big(\frac{-b}{c}\Big)\Big(\frac{c}{-b}\Big)=1,
\end{equation*}
which leads to
\begin{equation*}
\begin{split}
1=&\,(c, -1)_2(-a,-b)_2(-a,c)_2(-b,c)_2\\
 =&\,(c, -1)_2(-1, -1)_2(a,-1)_2(b,-1)_2(a,b)_2(a,c)_2(-1,c)_2(b,c)_2(-1,c)_2\\
 =&\,(-1,-1)_2s_2(f).
\end{split}
\end{equation*}

When one of $a,b,c$ is even, without loss of generality, assume that $c=2\tilde{c}$ is even.
By condition {\bf(B)}, there are integers $\tilde{u},\tilde{v},\tilde{w}$ such that
\begin{equation*}
-ab\equiv\tilde{u}^2~(\bmod~\tilde{c}),\,-b2\tilde{c}\equiv\tilde{v}^2~(\bmod-a)\,\,\text{and}\,-\tilde{c}2a\equiv\tilde{w}^2~(\bmod-b).
\end{equation*}
These congruence relations imply that
\begin{equation*}
\begin{split}
&\Big(\frac{-1}{\tilde{c}}\Big)\Big(\frac{-a}{\tilde{c}}\Big)\Big(\frac{-b}{\tilde{c}}\Big)=\Big(\frac{\tilde{u}^2}{\tilde{c}}\Big)=1,\\
&\Big(\frac{-b}{-a}\Big)\Big(\frac{2}{-a}\Big)\Big(\frac{\tilde{c}}{-a}\Big)=\Big(\frac{\tilde{v}^2}{-a}\Big)=1\\
&\,\text{and}\,\,\Big(\frac{-a}{-b}\Big)\Big(\frac{2}{-b}\Big)\Big(\frac{\tilde{c}}{-b}\Big)=\Big(\frac{\tilde{w}^2}{-b}\Big)=1.
\end{split}
\end{equation*}
Multiplying both sides of these three equations, it follows that
\begin{equation*}
\Big(\frac{-1}{\tilde{c}}\Big)\Big(\frac{2}{-a}\Big)\Big(\frac{2}{-b}\Big)\Big(\frac{-a}{-b}\Big)\Big(\frac{-b}{-a}\Big)
\Big(\frac{-a}{\tilde{c}}\Big)\Big(\frac{\tilde{c}}{-a}\Big)\Big(\frac{-b}{\tilde{c}}\Big)\Big(\frac{\tilde{c}}{-b}\Big)=1,
\end{equation*}
which leads to
\begin{equation*}
\begin{split}
1=&\,(\tilde{c},-1)_2(-a, 2)_2(-b, 2)_2(-a,-b)_2(-a,\tilde{c})_2(-b,\tilde{c})_2\\
 =&\,(2\tilde{c},-1)_2(-1,-1)_2(a,-1)_2(b,-1)_2(a,b)_2(-a,2\tilde{c})_2(-b,2\tilde{c})_2\\
 =&\,(-1, -1)_2(c,-1)_2(a,-1)_2(b,-1)_2(a,b)_2(a,c)_2(b,c)_2\\
 =&\,(-1,-1)_2s_2(f).
\end{split}
\end{equation*}

Therefore, $s_2(f)=-1=(-1,-1)_2$ and our proof is concluded.
\end{proof}

%%%%%%%%%%%%%%%%%%%%%%%%%%%%%%%%%%%%%%%%%%%%%%%%%%%%%%%%%%%%%%%%%%%%%%%%%%%%%%%%%%%%%%%%%%%%%%%%%%%%%%%%%%%%%%%%%%%%%%%%%%%%%%%%%%%%%%%%%%%%%%%%%%%%%%%%%%%%%%%%%%%%%%%%%%%%%%%%%
%%%%%%%%%%%%%%%%%%%%%%%%%%%%%%%%%%%%%%%%%%%%%%%%%%%%%%%%%%%%%%%%%%%%%%%%%%%%%%%%%%%%%%%%%%%%%%%%%%%%%%%%%%%%%%%%%%%%%%%%%%%%%%%%%%%%%%%%%%%%%%%%%%%%%%%%%%%%%%%%%%%%%%%%%%%%%%%%%
\section{Examples}\label{Sec:Exm}
In this section, we would like to use examples to show how one may apply Legendre's theorem to determine the existence of rational points on a given conic.

\begin{exm}
\begin{equation}\label{E4}
3X^2-12XY+17Y^2+18X-56Y+45=0.
\end{equation}
{\rm Step 1:} Replace $X$ by $\frac{X}{Z}$ and $Y$ by $\frac{Y}{Z}$ to get
\begin{equation*}
3X^2-12XY+17Y^2+18XZ-56YZ+45Z^2=0.
\end{equation*}
{\rm Step 2:} Diagonalize the resulted quadratic form by completing squares to get
\begin{equation*}
\begin{split}
&\,3X^2-12XY+17Y^2+18XZ-56YZ+45Z^2\\
=&\,3\big(X^2-4XY+6XZ+4Y^2-12YZ+9Z^2\big)+5Y^2-20YZ+18Z^2\\
=&\,3(X-2Y+3Z)^2+5\big(Y^2-4YZ+4Z^2\big)-2Z^2\\
=&\,3(X-2Y+3Z)^2+5(Y-2Z)^2-2Z^2.
\end{split}
\end{equation*}
{\rm Step 3:} Apply Legendre's theorem to determine if the equation $3x^2+5y^2-2z^2=0$ has a nonzero solution:
\vskip2pt
\noindent{\sc(i)} $a=3$, $b=5$ and $c=-2$ have different signs;
\vskip2pt
\noindent{\sc(ii)} $-ab=-15\equiv 1~(\bmod~c)$, $-ac=6\equiv 1~(\bmod~b)$ and $-bc=10\equiv 1~(\bmod~a)$ are all quadratic residues.
\vskip2pt
In conclusion, there are rational points on the conic \eqref{E4}.
\end{exm}

\begin{exm}
\begin{equation}\label{E5}
3X^2+6XY-2Y^2-18X-48Y-17=0.
\end{equation}
{\rm Step 1:} Replace $X$ by $\frac{X}{Z}$ and $Y$ by $\frac{Y}{Z}$ to get
\begin{equation*}
3X^2+6XY-2Y^2-18XZ-48YZ-17Z^2=0.
\end{equation*}
{\rm Step 2:} Diagonalize the resulted quadratic form by completing squares to get
\begin{equation*}
\begin{split}
&\,3X^2+6XY-2Y^2-18XZ-48YZ-17Z^2\\
=&\,3\big(X^2+2XY-6XZ+Y^2-6YZ+9Z^2\big)-5Y^2-30YZ-44Z^2\\
=&\,3(X+Y-3Z)^2-5\big(Y^2+6YZ+9Z^2\big)+Z^2\\
=&\,3(X+Y-3Z)^2-5(Y+3Z)^2+Z^2.
\end{split}
\end{equation*}
{\rm Step 3:} Apply Legendre's theorem to determine if the equation $3x^2-5y^2+z^2=0$ has a nonzero solution:
\vskip2pt
\noindent{\sc(i)} $a=3$, $b=-5$ and $c=1$ have different signs;
\vskip2pt
\noindent{\sc(ii)} $-ac=-3\equiv2~(\bmod-5)$, which is not a quadratic residue of $b$.
\vskip2pt
Therefore, there exists no rational point on the conic \eqref{E5}.
\end{exm}

%\begin{acknowledgments}
%\end{acknowledgments}

%%%%%%%%%%%%%%%%%%%%%%%%%%%%%%%%%%%%%%%%%%%%%%%%%%%%%%%%%%%%%%%%%%%%%%%%%%%%%%%%%%%%%%%%%%%%%%%%%%%%%%%%%%%%%%%%%%%%%%%%%%%%%%%%%%%%%%%%%%%%%%%%%%%%%%%%%%%%%%%%%%%%%%%%%%%%%%%%%
%%%%%%%%%%%%%%%%%%%%%%%%%%%%%%%%%%%%%%%%%%%%%%%%%%%%%%%%%%%%%%%%%%%%%%%%%%%%%%%%%%%%%%%%%%%%%%%%%%%%%%%%%%%%%%%%%%%%%%%%%%%%%%%%%%%%%%%%%%%%%%%%%%%%%%%%%%%%%%%%%%%%%%%%%%%%%%%%%

\end{document}